\documentclass[10pt]{article}

\usepackage{fullpage}
\usepackage{amsmath}
\usepackage{amsthm}
\usepackage{amsfonts}
\usepackage{amssymb}

\newcommand{\mb}{\mathbf}
\newcommand{\mc}{\mathcal}

\newtheorem{lemma}{Lemma}
\newtheorem{proposition}{Proposition}
\newtheorem{theorem}{Theorem}

\theoremstyle{remark}
\newtheorem{remark}{Remark}
\theoremstyle{definition}

\begin{document}

\title{Nonlinear stability of self--similar solutions for semilinear wave
equations}

\author{Roland Donninger\thanks{roland.donninger@univie.ac.at}\\
\small{Faculty of Physics, Gravitational Physics}\\
\small{University of Vienna}\\
\small{Boltzmanngasse 5}\\
\small{A-1090 Wien, Austria}}

\maketitle

\begin{abstract}
We prove nonlinear stability of the fundamental self--similar solution of the
wave equation with a focusing power nonlinearity $\psi_{tt}-\Delta \psi=\psi^p$
for $p=3,5,7,\dots$ in the radial case.
The proof is based on a semigroup formulation of the wave equation in 
similarity coordinates.
\end{abstract}

\section{Introduction}

\subsection{Motivation}

We study the nonlinear wave equation 
\begin{equation}
\label{eq_pwaveintro} \psi_{tt}-\Delta \psi=\psi^p
\end{equation}
where $\psi: \mathbb{R} \times \mathbb{R}^3 \to \mathbb{R}$ and 
$p>1$ is an odd integer.
The sign of the nonlinearity corresponds to the so--called focusing case, 
i.e. the equation shows a
tendency to magnify amplitudes which might eventually lead to singularity
formation. Indeed, there are explicit examples of solutions to Eq.
(\ref{eq_pwaveintro}) with smooth compactly supported initial data that 
blow up in finite time.
In order to show this one neglects the Laplacian and solves the resulting
ordinary differential equation in $t$. 
This yields the one--parameter family of solutions
$\psi^T(t,x)=c_0^{1/(p-1)}(T-t)^{-2/(p-1)}$ where $T>0$ and
$c_0=\frac{2(p+1)}{(p-1)^2}$.
We refer to $\psi^T$ as the \emph{fundamental self--similar solution}.
Although $\psi^T$ is homogeneous in space one can use smooth cut--off functions and
finite speed of propagation to construct a blow up solution with compactly
supported data.

In numerical evolutions \cite{bizon} for the radial equation 
one observes that generic and sufficiently 
large initial data lead to solutions that approach the fundmental self--similar
solution near the center $r=0$ for $t \to T-$. 
Thus, it is conjectured that the blow up described by $\psi^T$ is generic.
This conjecture is further supported by heuristic arguments 
(\cite{bizon}, \cite{pohozaev}), rigorous arguments on the linear
stability of $\psi^T$ (\cite{roland1}, \cite{roland2}) and a number of 
rigorous 
blow up results (\cite{merle1}, \cite{merle2}, \cite{merle3}).
Moreover, in \cite{merle4}, Merle and Zaag have studied the corresponding problem in one space dimension without symmetry assumptions and for arbitrary $p>1$.
It turns out that the fundamental self--similar solution $\psi^T$ is in fact a member of a more general family of explicit solutions that can be obtained by applying symmetry transformations (e.g. the Lorentz transform) to $\psi^T$.
Merle and Zaag have proved nonlinear stability of this family of solutions in the topology of the energy space, see Theorem 3 on p. 48 in \cite{merle4}.
Furthermore, the aforementioned authors have obtained important and deep results on the blow up curve for the one--dimensional problem \cite{merle5}, \cite{merle6}.

In the present paper we give a rigorous proof for the nonlinear stability (in a sense
to be made precise below) of the
fundamental self--similar solution $\psi^T$ in the radial case in dimension $3$.

\subsection{Overview}
The result is proved by using an operator formulation in 
similarity coordinates.
The coordinates $(\tau,\rho)$ are defined by $\tau:=-\log(T-t)$,
$\rho:=\frac{r}{T-t}$ and we restrict ourselves to $\rho \in (0,1)$ which
corresponds to the interior of the backward lightcone of the blow up point
$(t,r)=(T,0)$.
Thus, convergence to $\psi^T$ is shown only in this region.
In similarity coordinates, $t \to T-$ is equivalent to $\tau \to \infty$ 
and hence, we
are actually studying an asymptotic stability problem.
We rewrite Eq. (\ref{eq_pwaveintro}) as a first--order system in similarity
coordinates \footnote{We use the notation $L-\frac{2}{p-1}$ since this $L$ is exactly
the operator that has been studied in \cite{roland1} and \cite{roland2}.} 
\begin{equation}
\label{eq_cssintro}
\frac{d}{d\tau}\Phi(\tau)=\left (L-\frac{2}{p-1} \right )\Phi(\tau)+N(\Phi(\tau)) 
\end{equation}
where $\Phi$ consists of two components that are (roughly speaking) 
the time and space derivatives of nonlinear perturbations of $\psi^T$.
Hence, $\Phi \equiv 0$ corresponds to the fundamental self--similar solution of
Eq. (\ref{eq_pwaveintro}).
$L$ is a linear spatial differential operator that is realized as an
unbounded linear operator on an appropriate Hilbert space and, finally, $N$ is
the nonlinearity resulting from Eq. (\ref{eq_pwaveintro}).
The fundamental self--similar solution is asymptotically 
stable if any solution of Eq. (\ref{eq_cssintro}) with sufficiently small data
goes to zero as $\tau \to \infty$.
However, this cannot be quite true since there is an instability that emerges
from time translation symmetry. 
This instability comes from the fact that $\psi^T$ is a one--parameter 
family of
solutions rather than a single one.
Hence, the freedom of changing the blow up time is reflected by an unstable 
mode $\mb{g}$ of the linear operator $L-\frac{2}{p-1}$.
This mode is often referred to as the \emph{gauge mode}.
What we are actually interested in is stability modulo this symmetry, so--called
orbital stability.
On the linearized level, i.e. for the equation
\begin{equation} 
\label{eq_linearintro}
\frac{d}{d\tau}\Phi(\tau)
=\left (L-\frac{2}{p-1} \right )\Phi(\tau), 
\end{equation}
one defines an appropriate projection that removes the gauge 
instability from the
spectrum of $L-\frac{2}{p-1}$ and works on the stable subspace.
Then one can prove that any solution of Eq. (\ref{eq_linearintro}) decays as
$\tau \to \infty$ provided that $\Phi(0)$ belongs to the stable subspace.
Our result shows that this remains true on the nonlinear level in a certain
sense.
More precise, we prove that, given \emph{small} initial data $\mb{u}$, there
exists a constant $\alpha_\mb{u}$ such that Eq. (\ref{eq_cssintro}) has a unique
global solution $\Phi$ with initial data $\Phi(0)=\mb{u}+\alpha_\mb{u}\mb{g}$
and $\Phi$ decays for $\tau \to \infty$.
Moreover, the rate of decay is exactly given by the first stable mode of the linearized
operator $L-\frac{2}{p-1}$.
In other words, for any small perturbation $\mb{u}$
there exists a correction, which consists of adding a
multiple of the gauge mode, that leads to a time evolution that converges to the 
fundamental self--similar solution (i.e. to zero in this formulation) 
as $\tau \to \infty$.

This rigorous result corresponds to the following heuristic picture: 
If the data are small, the problem is essentially linear and it is possible 
to expand
the initial data in a sum of modes of the linearized operator.
Generic initial data will contain a contribution of the gauge mode and hence,
these data will lead to a solution that grows in time. 
However, by adding an appropriate multiple of the gauge mode to the data, 
it is possible to remove this instability and the resulting time evolution will
decay as $\tau \to \infty$.
 
The proof is based on a Banach iteration and it depends heavily on the good 
understanding of the linearized operator $L-\frac{2}{p-1}$ which has been
obtained in \cite{roland1} and \cite{roland2}.
The operator $L-\frac{2}{p-1}$ generates a strongly continuous semigroup
$\tilde{S}(\tau)$, i.e. the unique solution of the linearized problem Eq.
(\ref{eq_linearintro}) is given by $\Phi(\tau)=\tilde{S}(\tau)\Phi(0)$.
By applying the variation of constants formula we rewrite 
Eq. (\ref{eq_cssintro}) as an integral equation
\begin{equation}
\label{eq_intintro}
\Phi(\tau)=\tilde{S}(\tau)\Phi(0)+\int_0^\tau
\tilde{S}(\tau-\sigma)N(\Phi(\sigma))d\sigma
\end{equation}
and solve it by a fixed point iteration which is global in time.
To this end we define a mapping $K_\mb{u}$ by
$$ K_\mb{u}(\Phi)(\tau):=\tilde{S}(\tau)[\mb{u}+\alpha_\mb{u}(\Phi)\mb{g}]+
\int_0^\tau
\tilde{S}(\tau-\sigma)N(\Phi(\sigma))d\sigma $$
where $\mb{u}$ are the given (small) initial data and
$\alpha_\mb{u}(\Phi)\mb{g}$ is the correction which can be calculated by an 
explicit formula in terms of $\Phi$ and $\mb{u}$.
We study $K_\mb{u}$ on the Banach space 
$\mc{X}:=C([0,\infty), \mc{H}^{2k}) \cap
L^\infty([0,\infty),\mc{H}^{2k})$ with norm 
$$ \|\Phi\|_\mc{X}:=\sup_{\tau>0}\|\Phi(\tau)\|_{\mc{H}^{2k}} $$
where $\mc{H}^{2k}$ is an appropriate Sobolev space (integration with respect to
the spatial variable).
For $\delta>0$ we set 
$$ \mc{Y}_\delta:=\{\Phi \in \mc{X}: \|\Phi(\tau)\|_{\mc{H}^{2k}}\leq \delta
e^{-\tau} \mbox{ for all }\tau>0\} $$
and show that $\Phi \in \mc{Y}_\delta$ implies $K_\mb{u}(\Phi) \in
\mc{Y}_\delta$ provided that $\delta$ and $\|\mb{u}\|_{\mc{H}^{2k}}$ are
sufficiently small.
Under these smallness assumptions we further prove that $K_\mb{u}$ is a
contraction with respect to $\|\cdot\|_\mc{X}$ which yields the existence of a
fixed point in $\mc{Y}_\delta$ by the contraction mapping principle.
Thereby, we obtain the unique solution of Eq. (\ref{eq_cssintro}) with the desired
decay property.

\subsection{Additional remarks}
Let us briefly contrast our approach to the remarkable paper \cite{merle4} by Merle and Zaag.
The philosophy in \cite{merle4} is very different since the aforementioned authors study the behavior of \emph{any} blow up solution whereas we only consider small perturbations of $\psi^T$. 
Furthermore, the topology in \cite{merle4} is much weaker.
The proof of the impressive result in \cite{merle4} relies on the existence of a Lyapunov functional in similarity coordinates and therefore, the techniques used there are completely different to ours.
As a consequence, our approach yields independent and novel insights into self--similar blow up problems for nonlinear wave equations.
Although the result in \cite{merle4} is proved in dimension 1, the authors argue that the extension to higher dimensions $N$ is only technical. However, they have to require $1 < p \leq 1+\frac{4}{N-1}$, i.e., $1<p \leq 3$ for $N=3$.
It is clear that such a restriction does not exist in our approach.
In particular, we are able to cover the full energy supercritical range $p=7,9,11,\dots$ which has remained mostly unexplored so far.
We also note that our requirement of $p$ being an odd integer is a mere technicality to keep things as simple as possible.
In fact, one may equally apply our techniques to the problem Eq.~(\ref{eq_pwaveintro}) with the nonlinearity $|\psi|^{p-1}\psi$ for real $p>1$.
In order to treat the nonlinear term one would have to use results from \cite{merle4}, in particular the nonlinear estimate Claim 5.3 on p.~104 and the Hardy--Sobolev estimate of Lemma 2.2 on p.~51. 
Finally, the reader may recognize that some aspects of the present paper are inspired by 
the work
of Krieger and Schlag
\cite{schlag} (see also \cite{schlag2} for a survey) on the energy critical wave equation.

\subsection{Notation}
We write vectors as boldface letters and the components
are numbered by lower indices, e.g. $\mathbf{u}=(u_1,u_2)$.
For a Banach space $X$ we denote by $\mathcal{B}(X)$ the space of bounded linear
operators on $X$.
Throughout this work we use the symbols $\mc{H}$ and $\mc{H}^{2k}$ to denote the
Sobolev spaces $\mc{H}:=L^2(0,1) \times L^2(0,1)$ and
$$ \mc{H}^{2k}:=\{\mb{u} \in H^{2k}(0,1) \times H^{2k}(0,1): u_1^{(2j)}(0)=
u_2^{(2j+1)}(0)=0, j \in \mathbb{N}_0, j<k\} $$
for a $k \in \mathbb{N}$. 
We equip $\mc{H}$ and $\mc{H}^{2k}$ with the inner products
$$ (\mb{u}|\mb{v})_\mc{H}:=\sum_{j=1}^2 
\int_0^1 u_j(\rho)\overline{v_j(\rho)}d\rho $$
and
$$
(\mb{u}|\mb{v})_{\mc{H}^{2k}}:=(\mb{u}|\mb{v})_\mc{H}+\left (\mb{u}^{(2k)}
\left|\mb{v}^{(2k)}\right. \right )_\mc{H}. $$
Thus, $\mc{H}$ and $\mc{H}^{2k}$ are Hilbert spaces (cf. \cite{roland2}).
Finally, the expression $A \lesssim B$ means that there exists a $C>0$ such that
$A \leq CB$. 

\section{Derivation of the equations}
\subsection{Similarity coordinates and first--order formulation}
We consider the equation 
\begin{equation}
\label{eq_pwave}
\psi_{tt}-\Delta \psi=\psi^p
\end{equation} 
for $p=3,5,7,\dots$ in spherical symmetry.
The fundamental self--similar solution $\psi^T$ is given by 
$\psi^T(t,r)=c_0^{1/(p-1)}(T-t)^{-2/(p-1)}$ where $c_0=\frac{2(p+1)}{(p-1)^2}$
and $T>0$ is an arbitrary constant.
We are interested in small perturbations of $\psi^T$ and thus, we insert the
ansatz $\psi=\psi^T+\phi$ into Eq. (\ref{eq_pwave}) and apply the binomial
theorem to obtain
$$ \phi_{tt}-\Delta \phi=p(\psi^T)^{p-1}\phi+\sum_{j=2}^p \left (
\begin{array}{c}p \\ j \end{array} \right )(\psi^T)^{p-j}\phi^j. $$
With the substitution $\phi(t,r) \mapsto \tilde{\phi}(t,r):=r\phi(t,r)$, this
equation
transforms into
\begin{equation}
\label{eq_pwavetilde}
\tilde{\phi}_{tt}-\tilde{\phi}_{rr}=p(\psi^T)^{p-1}\tilde{\phi}
+r\sum_{j=2}^p \left (
\begin{array}{c}p \\ j \end{array} \right )(\psi^T)^{p-j}\left (
\frac{\tilde{\phi}}{r} \right )^j
\end{equation}
and we pick up the boundary condition $\tilde{\phi}(t,0)=0$ for all $t$.
Eq. (\ref{eq_pwavetilde}) is equivalent to the 
first--order system
\begin{multline}
\label{eq_pwavetilde1st}
\partial_t \left ( \begin{array}{c}\tilde{\phi}_t
\\ \tilde{\phi}_r \end{array} \right
)=\left ( \begin{array}{cc} 0 & \partial_r \\
\partial_r & 0 \end{array} \right )\left ( \begin{array}{c}\tilde{\phi}_t
\\ \tilde{\phi}_r \end{array} \right )+
\left ( \begin{array}{c} p (\psi^T)^{p-1} \int_0^r 
\tilde{\phi}_r(t,s)ds  \\ 0
\end{array} 
\right ) \\
+\sum_{j=2}^{p-1} \left (
\begin{array}{c}p \\ j \end{array} \right )
\left ( \begin{array}{c} r (\psi^T)^{p-j} \left( \frac{1}{r}\int_0^r 
\tilde{\phi}_r(t,s)ds \right )^j \\ 0
\end{array} 
\right ).
\end{multline}
Our aim is to study nonlinear perturbations of the fundamental self--similar
solution by using a formulation in similarity coordinates.
Appropriate similarity coordinates $(\tau,\rho)$ are given by
$\tau=-\log(T-t)$, $\rho=\frac{r}{T-t}$ and we restrict ourselves to the
interior of the backward lightcone of the blow up point $(t,r)=(T,0)$, that is
$\rho \in (0,1)$.
Eq. (\ref{eq_pwavetilde1st}) transforms into
\begin{multline}
\label{eq_pert}
\partial_\tau \left ( \begin{array}{c}\phi_1
\\ \phi_2 \end{array} \right
)=\left ( \begin{array}{cc}-\rho \partial_\rho-\frac{2}{p-1} & \partial_\rho \\
\partial_\rho & -\rho \partial_\rho-\frac{2}{p-1}
\end{array} \right )
\left ( \begin{array}{c}\phi_1 \\ \phi_2
\end{array} \right )
+\left ( \begin{array}{c} pc_0 \int_0^\rho \phi_2(\tau,\xi)d\xi \\ 0 \end{array} \right ) \\
+\sum_{j=2}^p \left ( \begin{array}{c}p \\ j 
\end{array} \right ) c_0^{\frac{p-j}{p-1}}\left ( \begin{array}{c} \rho  \left (\frac{1}{\rho} \int_0^\rho \phi_2(\tau,\xi)d\xi \right )^j \\ 0 
\end{array} \right )
\end{multline}
where
$\phi_1(\tau,\rho)=e^{-\frac{2}{p-1}\tau}
\tilde{\phi}_t(T-e^{-\tau},\rho e^{-\tau})$ and
$\phi_2(\tau,\rho)=e^{-\frac{2}{p-1}\tau}\tilde{\phi}_r(T-e^{-\tau},\rho 
e^{-\tau})$.

\subsection{Operator formulation}
We intend to formulate Eq. (\ref{eq_pert}) as an ordinary differential equation on the Hilbert space $\mc{H}$. To this end we define the operator $\tilde{L}: \mc{D}(\tilde{L}) \subset \mc{H} \to \mc{H}$ by 
$$ \mc{D}(\tilde{L}):=\{\mb{u} \in C^1[0,1] \times C^1[0,1]: u_1(0)=0\} $$ and
$$ \tilde{L}\mb{u}(\rho):=\left ( \begin{array}{c}-\rho u_1'(\rho)+u_2'(\rho)+pc_0 \int_0^\rho u_2(\xi)d\xi \\ u_1'(\rho)-\rho u_2'(\rho) \end{array} \right ). $$
An operator formulation of Eq. (\ref{eq_pert}) is given by
\begin{equation}
\label{eq_opnonlinearformal}
\frac{d}{d\tau}\Phi(\tau)=\left ( \tilde{L}-\frac{2}{p-1} \right )\Phi(\tau)+N(\Phi(\tau)) 
\end{equation}
where the nonlinearity $N$ is defined as
$$ N(\mb{u})(\rho)=\sum_{j=2}^p \left (\begin{array}{c} p \\ j \end{array} \right )
c_0^{\frac{p-j}{p-1}}\left ( \begin{array}{c} \rho \left ( \frac{1}{\rho}
\int_0^\rho u_2(\xi)d\xi \right )^j \\ 0 \end{array}
\right ).$$
Note that we are still on a formal level since 
$N(\mb{u}) \notin \mc{H}$ for general $\mb{u} \in \mc{H}$ 
as the example $\mb{u}(\rho)=(0,\rho^{-1/4})$ immediately shows.
However, after a more careful analysis of the nonlinearity $N$ in the 
next section, we will be able to turn Eq. (\ref{eq_opnonlinearformal}) into a 
well--defined operator differential equation.

\section{Formulation as an operator differential equation}

We analyse the nonlinearity $N$ in Eq. (\ref{eq_opnonlinearformal}) more 
carefully and formulate the problem we are going to study in a precise manner.

\subsection{Properties of the nonlinearity}
We need the following generalization of Hardy's inequality. 
The proof is
elementary but will be given in the appendix for the sake of completeness.

\begin{lemma}
\label{lem_hardy}
\begin{enumerate}
\item Let $v \in C^\infty[0,1]$ and $k \in \mathbb{N}$. Then
$$ \int_0^1 |v(x)|^2dx \lesssim \int_0^1 \frac{1}{x^{2k-2}}\left |
\frac{d}{dx}x^k v(x)\right |^2 dx. $$
\item Let $u \in C^\infty[0,1]$ with $u(0)=0$ and $j \in \mathbb{N}_0$. Then
$$ \int_0^1 \left |\frac{d^j}{dx^j}\frac{u(x)}{x}\right |^2dx \lesssim \int_0^1
|u^{(j+1)}(x)|^2 dx. $$
\end{enumerate}
\end{lemma}

\begin{proof}
See Appendix \ref{proof_lem_hardy}
\end{proof}

From now on we assume $k \in \mathbb{N}$ arbitrary but fixed.
The following lemma establishes two crucial estimates for the nonlinearity $N$.

\begin{lemma}
\label{lem_N}
$N$ defines a mapping from $\mc{H}^{2k}$ to itself.
Furthermore, we have the estimates
$$ \|N(\mb{u})\|_{\mc{H}^{2k}} \lesssim \sum_{j=2}^p \|\mb{u}\|_{\mc{H}^{2k}}^j
$$
and 
$$ \|N(\mb{u})-N(\mb{v})\|_{\mc{H}^{2k}} \lesssim
\|\mb{u}-\mb{v}\|_{\mc{H}^{2k}}\sum_{j=2}^p\sum_{\ell=0}^{j-1}
\|\mb{u}\|_{\mc{H}^{2k}}^{j-1-\ell}
\|\mb{v}\|_{\mc{H}^{2k}}^\ell $$
for all $\mb{u}, \mb{v} \in \mc{H}^{2k}$.
\end{lemma}

\begin{proof}
Let $\mathbf{u} \in \mathcal{H}^{2k}$ and define 
$\tilde{u}(\rho):=\frac{1}{\rho}\int_0^\rho u_2(\xi)d\xi$.
Then, 
$$ N(\mb{u})(\rho)=\sum_{j=2}^p \left (\begin{array}{c} p \\ j \end{array} \right )
c_0^{\frac{p-j}{p-1}}\left ( \begin{array}{c} \rho \tilde{u}(\rho)^j \\ 0 \end{array}
\right )$$ and by Lemma \ref{lem_hardy}, $\tilde{u} \in H^{2k}(0,1)$. 
We estimate
$\|\rho \mapsto \rho \tilde{u}(\rho)^j\|_{H^{2k}(0,1)} \lesssim 
\|\tilde{u}\|_{H^{2k}(0,1)}^j
$
since $H^{2k}(0,1)$ is a Banach algebra and Lemma \ref{lem_hardy} again implies
$\|\tilde{u}\|_{H^{2k}(0,1)}^j \lesssim \|u_2\|_{H^{2k}(0,1)}^j$ for $j=2,\dots,p$.
This shows
$$ \|N(\mb{u})\|_{\mc{H}^{2k}} \lesssim \sum_{j=2}^p \|\mb{u}\|_{\mc{H}^{2k}}^j.
$$ 
Note that $u_2^{(2\ell+1)}(0)=0$ for all $\ell < k$ which implies 
$\tilde{u}^{(2\ell+1)}(0)=0$ for all $\ell < k$.
The same holds true for $\tilde{u}^j$ and we conclude
$N(\mb{u}) \in \mc{H}^{2k}$.

To prove the second inequality we take $\mb{v} \in \mc{H}^{2k}$, define
$\tilde{v}(\rho):=\frac{1}{\rho}\int_0^\rho v_2(\xi)d\xi$ and note that 
$$ \tilde{u}^j-\tilde{v}^j=(\tilde{u}-\tilde{v})\sum_{\ell=0}^{j-1}
\tilde{u}^{j-1-\ell}\tilde{v}^\ell. $$
However, this already implies
$$ \|N(\mb{u})-N(\mb{v})\|_{\mc{H}^{2k}} \lesssim \|\mb{u}-\mb{v}\|_{\mc{H}^{2k}}
\sum_{j=2}^p \sum_{\ell=0}^{j-1}
\|\mb{u}\|_{\mc{H}^{2k}}^{j-1-\ell}\|\mb{v}\|_{\mc{H}^{2k}}^\ell $$
by the same reasoning as above.
\end{proof}

\subsection{The operator differential equation}
It is known that the operator $\tilde{L}$ is closable (see \cite{roland1}) and
we denote its closure by $L$. 
The nonlinear functional differential equation we are going to study is
\begin{equation}
\label{eq_opnonlinear}
\frac{d}{d\tau}\Phi(\tau)=\left ( L-\frac{2}{p-1} \right
)\Phi(\tau)+N(\Phi(\tau))
\end{equation}
where $\Phi: [0,\infty) \to \mc{H}$.
A function $\Phi: [0,\infty) \to \mc{H}$ 
is said to be an $\mc{H}^{2k}$--solution of 
Eq. (\ref{eq_opnonlinear}) with initial data $\mb{u} \in \mc{H}^{2k}$ if
\begin{itemize}
\item $\Phi(\tau) \in \mc{H}^{2k}$ for all $\tau>0$,
\item $\Phi$ is strongly differentiable in $\mc{H}^{2k}$, 
i.e. for any $\tau>0$ there exists an element
$\frac{d}{d\tau}\Phi(\tau) \in \mc{H}^{2k}$ such that
$$ \lim_{\sigma \to \tau}\left \|
\frac{\Phi(\tau)-\Phi(\sigma)}{\tau-\sigma}-\frac{d}{d\tau}\Phi(\tau)
\right \|_{\mc{H}^{2k}}=0,
$$
\item $\Phi(0)=\mb{u}$, 
\item $\Phi$ satisfies Eq. (\ref{eq_opnonlinear}) for all $\tau>0$.
\end{itemize}
We recall that Eq. (\ref{eq_opnonlinear}) is equivalent to the nonlinear wave
equation Eq. (\ref{eq_pwave}). The fundamental self--similar solution $\psi^T$
of Eq. (\ref{eq_pwave}) corresponds to the zero solution $\Phi \equiv 0$ for Eq.
(\ref{eq_opnonlinear}).

\section{The linearized operator}
The analysis of the nonlinear problem Eq. (\ref{eq_opnonlinear}) depends 
heavily on a good understanding of the linearization.
Thus, we review and extend some results of \cite{roland2} on the linearized
operator $L-\frac{2}{p-1}$.

\subsection{Spectral properties and growth estimates}
It is known (see \cite{roland2}) that the operator $L$ possesses a countable 
set of eigenvalues $\lambda_j^\pm$ with polynomial eigenfunctions $\mb{u}(\cdot,
\lambda_j^\pm)$ (that is, each component of $\mb{u}(\cdot,\lambda_j^\pm)$ is a
polynomial) where
$\lambda_j^+=1+\frac{2}{p-1}-2j$ and $\lambda_j^-=-\frac{2p}{p-1}-2j$.
Occasionally, we will refer to these eigenvalues as \emph{analytic eigenvalues}.
The single unstable eigenvalue $\lambda_0^+$ emerges from time
translation symmetry, i.e. the freedom of choosing the blow up time $T$ in the
definition of the similarity coordinates (see \cite{roland1} for a more thorough discussion).
Therefore, this instability is normally referred to as the \emph{gauge instability} since
it does not correspond to a "real" instability of $\psi^T$ but rather to a change
of the blow up time.
We have the following result from \cite{roland2}.

\begin{theorem}
\label{thm_S}
The operator $L$ generates a strongly continuous semigroup $S: [0,\infty) \to
\mc{B}(\mc{H})$ and the space $\mc{H}^{2k}$ is $L$--admissible (i.e.
$S(\tau)\mc{H}^{2k} \subset \mc{H}^{2k}$ and $S(\tau)|_{\mc{H}^{2k}}$ defines a
strongly continuous semigroup on $\mc{H}^{2k}$). For any $\mb{u} \in
\mc{H}^{2k}$ there exist constants $c_0^\pm, \dots, c_{k-1}^\pm \in \mathbb{C}$
and a function $\mb{f} \in \mc{H}^{2k}$ such that 
$$ \mb{u}=\sum_{j=0}^{k-1}\left (c_j^+ \mb{u}(\cdot, \lambda_j^+)+c_j^-
\mb{u}(\cdot, \lambda_j^-) \right ) + \mb{f} $$
and $\|S(\tau)\mb{f}\|_{\mc{H}^{2k}}\lesssim
e^{(\frac{1}{2}+pc_0-2k)\tau}\|\mb{f}\|_{\mc{H}^{2k}}$ for $\tau>0$ 
where $\mb{u}(\cdot,
\lambda_j^\pm)$ are normalized eigenfunctions of $L$ with eigenvalues 
$\lambda_j^+=1+\frac{2}{p-1}-2j$ and $\lambda_j^-=-\frac{2p}{p-1}-2j$.
\end{theorem}

\begin{remark}
\label{rem_N}
Additionally, we remark that the function $\mb{f}$ in Theorem \ref{thm_S} is
orthogonal (in $\mc{H}^{2k}$) to the $2k$ eigenfunctions $\mb{u}(\cdot,\lambda_j^\pm)$, 
$j=0,1,\dots,k-1$ (cf. \cite{roland2}).
However, it is important to note that the eigenfunctions $\mb{u}(\cdot, \lambda_j^\pm)$
are not orthogonal to each other since $L$ is not normal!
For brevity we denote the span of the $2k$ eigenfunctions by $\mc{N}$, i.e. 
$$ \mc{N}:=\langle \mb{u}(\cdot, \lambda_0^\pm), \mb{u}(\cdot, \lambda_1^\pm), \dots, 
\mb{u}(\cdot, \lambda_{k-1}^\pm) \rangle. $$
In \cite{roland2} it has been shown that the orthogonal complement 
$\mc{N}^\perp$ (in $\mc{H}^{2k}$) of the subspace $\mc{N}$
is invariant under $S(\tau)$ and the estimate
$\|S(\tau)\mb{f}\|_{\mc{H}^{2k}} \lesssim e^{(\frac{1}{2}+pc_0-2k)\tau}\|\mb{f}\|_{\mc{H}^{2k}}$, 
$\tau>0$,
is valid for all $\mb{f} \in \mc{N}^\perp$.
\end{remark}
 
 \subsection{Projection on the unstable subspace}
 
In what follows we denote the normalized eigenfunction $\mb{u}(\cdot,
\lambda_0^+)$ (the \emph{gauge mode}) by $\mb{g}$.
Our aim is to define a projection on the unstable subspace $\langle \mb{g}
\rangle$ that behaves nicely with respect to the time evolution generated by
$S(\tau)$.
First, we make the following easy observation.

\begin{lemma}
\label{lem_c0+unique}
The expansion coefficients $c_j^\pm$, $j=0,1,\dots, k-1$, in 
Theorem \ref{thm_S} are uniquely determined.
\end{lemma}

\begin{proof}
We denote by $Q \in \mc{B}(\mc{H}^{2k})$ the orthogonal projection on $\mc{N}$.
Projecting the expansion from Theorem \ref{thm_S} we obtain
$Q\mb{u}=\sum_{j=0}^{k-1}c_j^\pm \mb{u}(\cdot,\lambda_j^\pm)$
since $\mb{f} \in \mc{N}^\perp$ (Remark \ref{rem_N}).
The set $\{\mb{u}(\cdot, \lambda_j^\pm): j=0,\dots,k-1\}$ is linearly
independent and this implies the claim.
\end{proof}

To emphasize the dependence of the constants $c_j^\pm$ on $\mb{u}$ we write
$c_j^\pm(\mb{u})$.
Obviously, $c_j^\pm(\mb{u})$ is linear in $\mb{u}$.

\begin{lemma}
For any $j=0,1,\dots,k-1$ the mapping $\mb{u} \mapsto c_j^\pm(\mb{u}):
\mc{H}^{2k} \to \mathbb{C}$ is bounded.
\end{lemma}

\begin{proof}
Again, we denote by $Q \in \mc{B}(\mc{H}^{2k})$ the orthogonal projection on
$\mc{N}$. By definition and Remark \ref{rem_N} we have 
$c_j^\pm(\mb{u})=c_j^\pm(Q\mb{u})$.
However, $Q\mb{u}\mapsto c_j^\pm(Q\mb{u})$ is a linear mapping between the two
finite--dimensional Banach spaces $\mc{N}$ and $\mathbb{C}$ and hence, it is
bounded. We obtain 
$$ |c_j^\pm(\mb{u})|=|c_j^\pm(Q\mb{u})| \leq C
\|Q\mb{u}\|_{\mc{H}^{2k}} \leq C\|\mb{u}\|_{\mc{H}^{2k}} $$
for a $C>0$, any $\mb{u} \in \mc{H}^{2k}$ and all $j=0,1,\dots,k-1$.
\end{proof}

Thus, $\mb{u} \mapsto c_0^+(\mb{u})$ is a bounded linear functional on
$\mc{H}^{2k}$ and it follows from Riesz' theorem that there exists a 
$\mb{g}^* \in \mc{H}^{2k}$
such that $c_0^+(\mb{u})=(\mb{u}|\mb{g}^*)_{\mc{H}^{2k}}$ for all 
$\mb{u} \in \mc{H}^{2k}$.
By definition we have $(\mb{g}|\mb{g}^*)_{\mc{H}^{2k}}=1$.
We define the mapping $P: \mc{H}^{2k} \to
\mc{H}^{2k}$ by $P\mb{u}:=c_0^+(\mb{u})\mb{g}=(\mb{u}|\mb{g}^*)_{\mc{H}^{2k}}\mb{g}$.
It is clear that $P$ is linear and bounded.
Furthermore, we have $P^2=P$ and thus, $P$ is a projection on the closed subspace 
$\langle \mb{g} \rangle$ of $\mc{H}^{2k}$.
However, $P$ is not an orthogonal projection and hence not self--adjoint.

\begin{lemma}
\label{lem_SP}
The projection $P$ commutes with the semigroup $S$, i.e. $S(\tau)P\mb{u}=PS(\tau)\mb{u}$ for all 
$\tau>0$ and $\mb{u} \in \mc{H}^{2k}$.
\end{lemma}

\begin{proof}
Fix $\tau >0$ and let $\mb{u} \in \mc{H}^{2k}$.
Invoking Theorem \ref{thm_S} we obtain
$$ \mb{u}=\sum_{j=0}^{k-1}c_j^\pm \mb{u}(\cdot, \lambda_j^\pm)+\mb{f} $$
where $\mb{f} \in \mc{N}^\perp$ and applying $S(\tau)$ yields
$$ S(\tau)\mb{u}=\sum_{j=0}^{k-1}c_j^\pm e^{\lambda_j^\pm \tau}\mb{u}(\cdot, \lambda_j^\pm)
+S(\tau)\mb{f}. $$
Note that this is an expansion of $S(\tau)\mb{u}$ in the sense of 
Theorem \ref{thm_S} since $S(\tau)\mb{f} \in \mc{N}^\perp$ by Remark \ref{rem_N}.
Thus, by definition of $P$ we have $PS(\tau)\mb{u}=c_0^+e^{\lambda_0^+ \tau}\mb{g}$.

On the other hand, we have $S(\tau)P\mb{u}=S(\tau)c_0^+ \mb{g}=c_0^+ e^{\lambda_0^+ \tau}\mb{g}$.
\end{proof}

\subsection{Properties of $\tilde{S}$}
Actually we are interested in the semigroup $\tilde{S}$ defined by $\tilde{S}(\tau)=e^{-\frac{2}{p-1}\tau}S(\tau)$ which yields the solution of the linearized equation
$$ \frac{d}{d\tau}\Phi(\tau)=\left (L-\frac{2}{p-1} \right ) \Phi(\tau). $$
But $\tilde{S}$ is only a trivial rescaling of $S$ and so we can immediately deduce important properties.
First of all we remark that the projection $P$ commutes with $\tilde{S}(\tau)$ (Lemma \ref{lem_SP}).
Furthermore, the gauge mode is an eigenfunction of $L-\frac{2}{p-1}$ with eigenvalue $1$ and thus,
we have $\tilde{S}(\tau)\mb{g}=e^\tau\mb{g}$.
Appropriate growth estimates are given in the following proposition.

\begin{proposition}
\label{prop_S}
If $k \in \mathbb{N}$ is sufficiently large, 
the semigroup $\tilde{S}$ satisfies the estimates
$$ \|\tilde{S}(\tau)\mb{u}\|_{\mc{H}^{2k}} \lesssim e^{\tau}\|\mb{u}\|_{\mc{H}^{2k}} $$
and
$$  \|\tilde{S}(\tau)(I-P)\mb{u}\|_{\mc{H}^{2k}} \lesssim e^{-\tau}\|(I-P)\mb{u}\|_{\mc{H}^{2k}} $$
for all $\mb{u} \in \mc{H}^{2k}$ and $\tau>0$.
\end{proposition}

\begin{proof}
The estimates are immediate consequences of Theorem \ref{thm_S} and the fact that the largest 
analytic eigenvalue of $L$ apart from $\lambda_0^+$ is $\lambda_1^+=\frac{2}{p-1}-1$.
\end{proof}

In what follows we implicitly assume $k$ to be so large that Proposition
\ref{prop_S} holds.

\section{Global existence for the nonlinear problem}
\label{sec_globnl}

Our intention is to prove existence for Eq. (\ref{eq_opnonlinear}) by
means of a Banach iteration which is global in time.

\subsection{Function spaces}
Let $\mc{X}:=L^\infty([0,\infty), \mc{H}^{2k}) \cap C([0,\infty),\mc{H}^{2k})$ 
and set
$\|\Phi\|_\mc{X}:=\sup_{\tau>0}\|\Phi(\tau)\|_{\mc{H}^{2k}}$.
$\mc{X}$ equipped with $\|\cdot\|_\mc{X}$ is a Banach space.
We define the closed subset $\mc{Y}_\delta \subset \mc{X}$ by
$$ \mc{Y}_\delta:=\{\Phi \in \mc{X}:
\|\Phi(\tau)\|_{\mc{H}^{2k}}\leq \delta
e^{-\tau} \mbox{ for all }\tau>0\} $$
where $\delta>0$.
As the following lemma shows, the nonlinearity $N$ behaves well on 
$\mc{Y}_\delta$ provided that $\delta$ is
chosen small enough.

\begin{lemma}
\label{lem_propN}
If $\delta \leq 1$ then there exists a constant $c>0$ such that 
$$\|N(\Phi(\tau))\|_{\mc{H}^{2k}}\leq c \delta^2 e^{-2\tau}$$ and
$$\|N(\Phi(\tau))-N(\Psi(\tau))\|_{\mc{H}^{2k}}\leq c \delta
e^{-\tau}\|\Phi(\tau)-\Psi(\tau)\|_{\mc{H}^{2k}}$$
for all $\Phi, \Psi \in \mc{Y}_\delta$ and $\tau>0$.
\end{lemma}

\begin{proof}
Let $\Phi \in \mc{Y}_\delta$.
Lemma \ref{lem_N} implies the existence of a $c_1>0$ such that 
$$ \|N(\Phi(\tau))\|_{\mc{H}^{2k}} \leq c_1 \sum_{j=2}^p
\|\Phi(\tau)\|_{\mc{H}^{2k}}^j \leq c_1 \sum_{j=2}^p \left(\delta 
e^{-\tau} \right )^j. $$
Since $\delta \leq 1$ we have $\delta e^{-\tau} \leq 1$ for all $\tau>0$ and 
thus, $(\delta e^{-\tau})^j \leq \delta^2 e^{-2\tau}$ for $j=2,\dots,p$,
$\tau>0$.
This implies $\|N(\Phi(\tau))\|_{\mc{H}^{2k}} \leq (p-1)c_1 \delta^2 e^{-2\tau}$
which is the first inequality in the claim.

Let $\Psi \in \mc{Y}_\delta$ and apply Lemma \ref{lem_N} to obtain
$$\|N(\Phi(\tau))-N(\Psi(\tau))\|_{\mc{H}^{2k}}\leq c_2
\|\Phi(\tau)-\Psi(\tau)\|_{\mc{H}^{2k}}
\sum_{j=2}^p\sum_{\ell=0}^{j-1}
\|\Phi(\tau)\|_{\mc{H}^{2k}}^{j-1-\ell}
\|\Psi(\tau)\|_{\mc{H}^{2k}}^\ell $$ 
for a constant $c_2>0$.
Since $\delta \leq 1$ we observe that $\|\Phi(\tau)\|_{\mc{H}^{2k}}^{j-1-\ell}
\|\Psi(\tau)\|_{\mc{H}^{2k}}^\ell \leq (\delta e^{-\tau})^{j-1} \leq 
\delta e^{-\tau}$ for $j=2,\dots,p$, $\tau>0$ and
this implies the second assertion.
\end{proof}

\subsection{The contraction mapping}
For fixed $\mb{u} \in \mc{H}^{2k}$ with $\|\mb{u}\|_{\mc{H}^{2k}}\leq
\delta^2$ and $0 < \delta \leq 1$ 
we define the nonlinear mapping $K_\mb{u}: \mc{X} \to C([0,\infty),
\mc{H}^{2k})$ by
$$ K_\mb{u}(\Phi)(\tau):=\tilde{S}(\tau)[\mb{u}+\alpha_\mb{u}(\Phi)\mb{g}]+\int_0^\tau
\tilde{S}(\tau-\sigma)N(\Phi(\sigma))d\sigma $$
where $\alpha_\mb{u}(\Phi) \in \mathbb{C}$ is given by
$$ \alpha_\mb{u}(\Phi):=-\int_0^\infty e^{-\sigma}
(N(\Phi(\sigma))|\mb{g}^*)_{\mc{H}^{2k}}d\sigma-
(\mb{u}|\mb{g}^*)_{\mc{H}^{2k}}. $$
The integral in the definition of $K_\mb{u}$ has to be interpreted as a Riemann integral
over a continuous function with values in $\mc{H}^{2k}$.
Note that the integrals above exist
since $\Phi \in \mc{X}$ implies
$\|N(\Phi(\sigma))\|_{\mc{H}^{2k}} \lesssim 1$ for all $\sigma>0$ (cf. Lemma
\ref{lem_N}).

A fixed point $\Phi$ of $K_\mb{u}$ (i.e. $\Phi=K_\mb{u}(\Phi)$) satisfies the
equation
\begin{equation}
\label{eq_opnonlinearint} \Phi(\tau)=\tilde{S}(\tau)[\mb{u}+\alpha_\mb{u}(\Phi)\mb{g}]+\int_0^\tau
\tilde{S}(\tau-\sigma)N(\Phi(\sigma))d\sigma 
\end{equation}
and this is an integral formulation of Eq. (\ref{eq_opnonlinear}) with
initial data $\Phi(0)=\mb{u}+\alpha_\mb{u}(\Phi)\mb{g}$.

\begin{proposition}
\label{prop_Yinv}
If $\delta$ is sufficiently small then $\Phi \in \mc{Y}_\delta$ implies $K_\mb{u}(\Phi)
\in \mc{Y}_\delta$.
\end{proposition}

\begin{proof}
Let $\Phi \in \mc{Y}_\delta$.
We decompose $K_\mb{u}(\Phi)(\tau)=PK_\mb{u}(\Phi)(\tau)+(I-P)K_\mb{u}(\Phi)(\tau)$ and analyse the
two parts separately.
By taking the inner product of $PK_\mb{u}(\Phi)(\tau)$ with $\mb{g}^*$ we obtain
$$ (PK_\mb{u}(\Phi)(\tau)|\mb{g}^*)_{\mc{H}^{2k}}=\left
( \left. \tilde{S}(\tau)[P\mb{u}+\alpha_\mb{u}(\Phi)\mb{g}] \right | \mb{g}^*
 \right )_{\mc{H}^{2k}}
+\int_0^\tau \left (\left. \tilde{S}(\tau-\sigma)PN(\Phi(\sigma))\right | \mb{g}^*
 \right )_{\mc{H}^{2k}}d\sigma $$
where $\tilde{S}(\tau)P=P\tilde{S}(\tau)$ and the continuity of the inner product has been used.
Since $P\mb{\mb{f}}=(\mb{\mb{f}}|\mb{g}^*)_{\mc{H}^{2k}}\mb{g}$ for any 
$\mb{\mb{f}} \in
\mc{H}^{2k}$, $(\mb{g}|\mb{g}^*)_{\mc{H}^{2k}}=1$ and 
$\tilde{S}(\tau)\mb{g}=e^\tau \mb{g}$ we infer
$$ \|PK_\mb{u}(\Phi)(\tau)\|_{\mc{H}^{2k}}=\left |
e^{\tau}(\mb{u}|\mb{g}^*)_{\mc{H}^{2k}}+e^{\tau}\alpha_\mb{u}(\Phi)+\int_0^\tau
e^{\tau-\sigma}(N(\Phi(\sigma))|\mb{g}^*)_{\mc{H}^{2k}}d\sigma \right |. $$
Inserting the definition of $\alpha_\mb{u}(\Phi)$ leads to
$$ \|PK_\mb{u}(\Phi)(\tau)\|_{\mc{H}^{2k}}=\left |\int_\tau^\infty
e^{\tau-\sigma}(N(\Phi(\sigma))|\mb{g}^*)_{\mc{H}^{2k}}d\sigma \right | $$
and Lemma \ref{lem_propN} implies
$$ \|PK_\mb{u}(\Phi)(\tau)\|_{\mc{H}^{2k}} \leq c\delta^2 
\int_\tau^\infty e^{\tau-3\sigma}d\sigma=\frac{c}{3}\delta^2 e^{-2\tau} \leq
\frac{\delta}{2}e^{-\tau} $$  
for a $c>0$ and all $\tau>0$ provided that $\delta \leq 
\min\{1, \frac{3}{2c}\}$.

For the infinite--dimensional part we obtain
$$ \|(I-P)K_\mb{u}(\Phi)(\tau)\|_{\mc{H}^{2k}}\lesssim
e^{-\tau}\|\mb{u}\|_{\mc{H}^{2k}}+\int_0^\tau
e^{-\tau+\sigma}\|N(\Phi(\sigma))\|_{\mc{H}^{2k}}d\sigma $$
by Proposition \ref{prop_S} and therefore, Lemma \ref{lem_propN} implies
$$ \|(I-P)K_\mb{u}(\Phi)(\tau)\|_{\mc{H}^{2k}}\lesssim \delta^2 e^{-\tau}+\delta^2 
\int_0^\tau e^{-\tau-\sigma} d\sigma \leq 2 \delta^2 e^{-\tau}. $$ 
Hence, there exists a constant $\tilde{c}>0$ such that
$$ \|(I-P)K_\mb{u}(\Phi)(\tau)\|_{\mc{H}^{2k}} \leq \tilde{c}\delta^2 e^{-\tau} $$
and, if $\delta \leq \frac{1}{2\tilde{c}}$, we arrive at 
$\|(I-P)K_\mb{u}(\Phi)(\tau)\|_{\mc{H}^{2k}} \leq \frac{\delta}{2}e^{-\tau}$ for all
$\tau>0$ and the claim is proved with 
$\delta \leq \min\{1, \frac{3}{2c}, \frac{1}{2\tilde{c}}\}$.
\end{proof}

\begin{proposition}
\label{prop_contr}
If $\delta$ is sufficiently small then we have the estimate 
$$\|K_\mb{u}(\Phi)-K_\mb{u}(\Psi)\|_\mc{X}\leq \frac{1}{2}\|\Phi-\Psi\|_\mc{X}$$ 
for all $\Phi, \Psi \in \mc{Y}_\delta$.
\end{proposition}

\begin{proof}
Let $\Phi,\Psi \in \mc{Y}_\delta$ and consider the finite--dimensional part
$PK_\mb{u}(\Phi)(\tau)-PK_\mb{u}(\Psi)(\tau)$ first.
Pairing with $\mb{g}^*$ we obtain
\begin{multline*} \|PK_\mb{u}(\Phi)(\tau)
-PK_\mb{u}(\Psi)(\tau)\|_{\mc{H}^{2k}}\\
=\left |
[\alpha_\mb{u}(\Phi)-\alpha_\mb{u}(\Psi)]e^{\tau}+\int_0^\tau
e^{\tau-\sigma}(N(\Phi(\sigma))-N(\Psi(\sigma))|\mb{g}^*)_{\mc{H}^{2k}}d\sigma
\right | \\
= \left | \int_\tau^\infty e^{\tau-\sigma}
(N(\Phi(\sigma))-N(\Psi(\sigma))|\mb{g}^*)_{\mc{H}^{2k}}d\sigma
\right |\leq c\delta \int_\tau^\infty
e^{\tau-2\sigma}\|\Phi(\sigma)-\Psi(\sigma)\|_{\mc{H}^{2k}}d\sigma \\
\leq c \delta
\sup_{\sigma>0}\|\Phi(\sigma)-\Psi(\sigma)\|_{\mc{H}^{2k}}\int_\tau^\infty
e^{\tau-2\sigma}d\sigma=\frac{c}{2}\delta e^{-\tau}\|\Phi-\Psi\|_\mc{X}
\end{multline*}
for all $\tau>0$ by Lemma \ref{lem_propN}.
If $\delta \leq \frac{1}{2c}$ we arrive at
$$ \sup_{\tau>0}\|PK_\mb{u}(\Phi)(\tau)-PK_\mb{u}(\Psi)(\tau)\|_{\mc{H}^{2k}} 
\leq \frac{1}{4} \|\Phi-\Psi\|_\mc{X}. $$

For the infinite--dimensional part we have
\begin{multline*}
\|(I-P)K_\mb{u}(\Phi)(\tau)-(I-P)K_\mb{u}(\Psi)(\tau)\|_{\mc{H}^{2k}} \leq \int_0^\tau
\|\tilde{S}(\tau-\sigma)(I-P)[N(\Phi(\sigma))-N(\Psi(\sigma))]\|_{\mc{H}^{2k}}d\sigma\\
\lesssim \int_0^\tau
e^{-\tau+\sigma}\|N(\Phi(\sigma))-N(\Psi(\sigma))\|_{\mc{H}^{2k}}d\sigma
\lesssim \delta \int_0^\tau e^{-\tau}
\|\Phi(\sigma)-\Psi(\sigma)\|_{\mc{H}^{2k}}d\sigma \\
\leq \delta \tau e^{-\tau} \sup_{\sigma>0}
\|\Phi(\sigma)-\Psi(\sigma)\|_{\mc{H}^{2k}}
\end{multline*}
by Lemma \ref{lem_propN} again.
However, with $\delta$ small enough this implies
$$ \sup_{\tau>0}\|(I-P)K_\mb{u}(\Phi)(\tau)-(I-P)K_\mb{u}(\Psi)(\tau)\|_{\mc{H}^{2k}} 
\leq \frac{1}{4}
\|\Phi-\Psi\|_\mc{X} $$
and we arrive at the claim.
\end{proof}

\subsection{Global existence and uniqueness of the solution}
Propositions \ref{prop_Yinv} and \ref{prop_contr} show that there exists a $\delta>0$ 
such that, if $\|\mb{u}\|_{\mc{H}^{2k}} \leq \delta^2$, 
the mapping $K_\mb{u}$ restricted to $\mc{Y}_\delta$ has range in $\mc{Y}_\delta$ 
and is a contraction
with respect to $\|\cdot\|_\mc{X}$. Since $\mc{Y}_\delta \subset \mc{X}$ is closed, 
the contraction mapping principle yields the existence of a unique 
fixed point of $K_\mb{u}$ in $\mc{Y}_\delta$.
In fact, the fixed point is unique in the whole space $\mc{X}$, as the following standard
argument shows.

\begin{lemma}
Let $\Phi \in \mc{Y}_\delta$, $\Psi \in \mc{X}$ be fixed points of $K_\mb{u}$ 
with $\Phi(0)=\Psi(0)$. Then $\Phi=\Psi$.
\end{lemma}

\begin{proof}
Fix $\tau_0>0$.
The function $\Phi-\Psi$ satisfies the integral equation
$$ \Phi(\tau)-\Psi(\tau)=\int_0^\tau \tilde{S}(\tau-\sigma)[N(\Phi(\sigma))-N(\Psi(\sigma))]d\sigma $$
and hence, for all $\tau \in [0,\tau_0]$, we have
\begin{multline*}  \|\Phi(\tau)-\Psi(\tau)\|_{\mc{H}^{2k}} \leq \int_0^\tau e^{\tau-\sigma}
\|N(\Phi(\sigma))-N(\Psi(\sigma))\|_{\mc{H}^{2k}}d\sigma \\
\leq \tau e^\tau 
\sup_{\sigma \in (0,\tau)} \|N(\Phi(\sigma))-N(\Psi(\sigma))\|_{\mc{H}^{2k}}\\
\leq \tau e^\tau C M(\tau_0) \sup_{\sigma \in (0,\tau)} \|\Phi(\sigma)-\Psi(\sigma)\|_{\mc{H}^{2k}}
\end{multline*}
by Lemma \ref{lem_N} where 
$$M(\tau_0):=
\sup_{\sigma \in (0,\tau_0)} \sum_{j=2}^p \sum_{\ell=0}^{j-1}
\|\Phi(\sigma)\|_{\mc{H}^{2k}}^{j-1-\ell}\|\Psi(\sigma)\|_{\mc{H}^{2k}}^\ell < \infty$$
and $C>0$.
Thus, there exists a $\tau_1 \in (0,\tau_0]$ such that
$$ \sup_{\tau \in (0,\tau_1)}\|\Phi(\tau)-\Psi(\tau)\|_{\mc{H}^{2k}} \leq 
\frac{1}{2}\sup_{\tau \in (0,\tau_1)}\|\Phi(\tau)-\Psi(\tau)\|_{\mc{H}^{2k}} $$
and this implies $\Phi(\tau)=\Psi(\tau)$ for all $\tau \in [0,\tau_1]$.
Iterating this argument we obtain $\Phi(\tau)=\Psi(\tau)$ for all $\tau \in [0,\tau_0]$ and, 
since $\tau_0>0$ was arbitrary, we conclude $\Phi=\Psi$.
\end{proof}

\section{The main theorem}
We denote by $L_{\mc{H}^{2k}}$ the part of $L$ in $\mc{H}^{2k}$, i.e. the
operator $L_{\mc{H}^{2k}}: \mc{D}({\mc{H}^{2k}}) \subset \mc{H}^{2k} \to
\mc{H}^{2k}$ defined by 
$$ \mc{D}({\mc{H}^{2k}}):=\{\mb{u} \in \mc{D}(L) \cap \mc{H}^{2k}: L\mb{u} \in
\mc{H}^{2k}\} $$
and $L_{\mc{H}^{2k}} \mb{u}:=L \mb{u}$.
Note that $L_{\mc{H}^{2k}}$ is densely defined since $\mc{H}^{2(k+1)} \subset
\mc{D}(L_{\mc{H}^{2k}})$ (cf. \cite{roland2}).
Now we are ready to formulate and prove our main result.

\begin{theorem}
\label{thm_nlstabop}
Let $k \in \mathbb{N}$ be sufficiently large, $\delta>0$ sufficiently small.
Then, for any 
$\mb{u} \in \mc{D}(L_{\mc{H}^{2k}})$ with 
$\|\mb{u}\|_{\mc{H}^{2k}} \leq \delta^2$,
there exists an $\alpha_\mb{u} \in \mathbb{C}$ such that the equation
$$ \frac{d}{d\tau}\Phi(\tau)=\left ( L-\frac{2}{p-1} \right )\Phi(\tau)+N(\Phi(\tau)) $$
has a unique global $\mc{H}^{2k}$--solution $\Phi$ with initial data 
$\Phi(0)=\mb{u}+\alpha_\mb{u}\mb{g}$ 
that satisfies $\|\Phi(\tau)\|_{\mc{H}^{2k}} \lesssim e^{-\tau}$ for all 
$\tau>0$.
\end{theorem}

\begin{proof}
By Theorem \ref{thm_S}, $\tilde{S}(\tau)|_{\mc{H}^{2k}}$ defines a 
semigroup on $\mc{H}^{2k}$ and its generator is $L_{\mc{H}^{2k}}-\frac{2}{p-1}$.
Furthermore, $\mb{g} \in \mc{D}(L_{\mc{H}^{2k}})$ and thus, $\tau \mapsto
\tilde{S}(\tau)\mb{v}$, where $\mb{v}:=\mb{u}+\alpha_\mb{u}(\Phi)\mb{g} \in
\mc{D}(L_{\mc{H}^{2k}})$, is strongly differentiable in $\mc{H}^{2k}$ and we have
$\frac{d}{d\tau}\tilde{S}(\tau)\mb{v}=(L_{\mc{H}^{2k}}-\frac{2}{p-1})
\tilde{S}(\tau)\mb{v}$.
According to the results of Sec. \ref{sec_globnl} there exists a unique
(in $\mc{X}$) function
$\Phi \in \mc{Y}_\delta$ that satisfies
$$ \Phi(\tau)=\tilde{S}(\tau)\mb{v}+\int_0^\tau
\tilde{S}(\tau-\sigma)N(\Phi(\sigma))d\sigma. $$
Differentiating this equation (with respect to $\|\cdot\|_{\mc{H}^{2k}}$) we
obtain
\begin{multline*} \frac{d}{d\tau}\Phi(\tau)=\left (L_{\mc{H}^{2k}}-\frac{2}{p-1} \right )\left
(\tilde{S}(\tau)\mb{v}+\int_0^\tau \tilde{S}(\tau-\sigma)N(\Phi(\sigma)) \right
) + N(\Phi(\tau))\\
= \left (L_{\mc{H}^{2k}}-\frac{2}{p-1} \right )\Phi(\tau)+N(\Phi(\tau))
=\left (L-\frac{2}{p-1} \right )\Phi(\tau)+N(\Phi(\tau))
\end{multline*}
where we have interchanged the operator $L_{\mc{H}^{2k}}-\frac{2}{p-1}$ and the
integral sign which is justified by the closedness of 
$L_{\mc{H}^{2k}}-\frac{2}{p-1}$.

\end{proof}
In other words, Theorem \ref{thm_nlstabop} tells us that, given sufficiently 
regular and small data, there exists a "correction"
of the data (which consists of adding a multiple of the gauge mode) 
that leads to a global solution that goes to zero as $\tau \to \infty$. 
The correction of the data corresponds exactly to what is called
"tuning out" the gauge instability in the
heuristic picture.  

Finally, we remark that the required 
degree of differentiability $k$ in Theorem
\ref{thm_nlstabop} could be specified more explicitly. To this end one would
have to optimize the results of \cite{roland2} which is certainly possible.

\section{Acknowledgments}
The author would like to thank Peter C. Aichelburg and Nikodem Szpak for 
helpful discussions and their interest in this work.

\begin{appendix}
\section{Proof of Lemma \ref{lem_hardy}}
\label{proof_lem_hardy}
\begin{enumerate}
\item Integration by parts and the Cauchy--Schwarz inequality yield
\begin{multline*}
\int_0^1 |v(x)|^2 dx = \int_0^1
\frac{1}{x^{2k}}|x^k v(x)|^2dx \\=\left. -\frac{|x^{k}v(x)|^2}{(2k-1)x^{2k-1}}
\right |_0^1+\frac{2}{2k-1} \mathrm{Re} \int_0^1 
\frac{x^k}{x^{2k-1}} \overline{v(x)} \frac{d}{dx}(x^k v(x))dx\\
\lesssim \left (\int_0^1 |v(x)|^2 dx \right )^{1/2}\left (\int_0^1
\frac{1}{x^{2k-2}}\left |\frac{d}{dx}x^k v(x)\right |^2dx \right )^{1/2} 
\end{multline*}
\item We have
$$ \frac{d^j}{dx^j}\frac{u(x)}{x}=\frac{1}{x^{j+1}}\sum_{\ell=0}^j (-1)^\ell
\ell ! \left ( \begin{array}{c}j \\ \ell \end{array} \right )x^{j-\ell}
u^{(j-\ell)}(x)
$$
and thus,
\begin{multline*}
\frac{d}{dx}x^{j+1}\left (\frac{d^j}{dx^j}\frac{u(x)}{x} \right
)=\sum_{\ell=0}^{j-1}(-1)^\ell \frac{j!}{(j-\ell-1)!}x^{j-\ell-1}
u^{(j-\ell)}(x)\\
+\underbrace{
\sum_{\ell=0}^j (-1)^\ell \frac{j!}{(j-\ell)!}x^{j-\ell}u^{(j-\ell+1)}(x)}_
{\sum_{\ell=-1}^{j-1} (-1)^{\ell+1} \frac{j!}{(j-\ell-1)!}x^{j-\ell-1}
u^{(j-\ell)}(x)}=x^j u^{(j+1)}(x)
\end{multline*}
Applying part 1 with $v(x)=\frac{d^j}{dx^j}\frac{u(x)}{x}$ and $k=j+1$ yields the claim.
\end{enumerate}
\end{appendix}

\bibliography{nonlinear}{}
\bibliographystyle{plain}

\end{document}